\documentclass{amsart}

\usepackage{mathrsfs}
\usepackage{stmaryrd}
\usepackage{amscd}
\usepackage{amssymb}
\usepackage[alphabetic]{amsrefs}
\usepackage[all]{xy}
\usepackage{mathtools}

\title[Eichler-Shimura relations for local Shimura varieties]
{Eichler-Shimura relations for local Shimura varieties}
\author{Teruhisa Koshikawa}
\address{Research Institute for Mathematical Sciences, Kyoto University}
\email{teruhisa@kurims.kyoto-u.ac.jp}

\theoremstyle{plain}
\newtheorem{thm}{Theorem}[section]

\newtheorem{prop}[thm]{Proposition}
\newtheorem{cor}[thm]{Corollary}

\theoremstyle{definition}
\newtheorem{rem}[thm]{Remark}

\newcommand\bA{\mathbb A}

\newcommand\bC{\mathbf C}
\newcommand\bF{\mathbf F}
\newcommand\bQ{\mathbf Q}
\newcommand\bZ{\mathbf Z}

\newcommand\cC{\mathcal C}
\newcommand\cD{\mathcal D}
\newcommand\cG{\mathcal G}
\newcommand\cM{\mathcal M}

\newcommand\cO{\mathcal O}
\newcommand\cV{\mathcal V}
\newcommand\cZ{\mathcal Z}

\DeclareMathOperator{\GL}{GL}		
\DeclareMathOperator{\Frob}{Frob}
\DeclareMathOperator{\End}{End}
\DeclareMathOperator{\Gal}{Gal}

\DeclareMathOperator{\Hom}{Hom}
\DeclareMathOperator{\Rep}{Rep}
\DeclareMathOperator{\id}{id}

\DeclareMathOperator{\cInd}{c-Ind}
\DeclareMathOperator{\Bun}{Bun}

\DeclareMathOperator{\Sh}{Sh}
\DeclareMathOperator{\Ext}{Ext}

\DeclareMathOperator{\op}{op}
\DeclareMathOperator{\lis}{lis}
\DeclareMathOperator{\Perf}{Perf}

\begin{document}

\begin{abstract}
In this note, we discuss several forms of the Eichler-Shimura relation for the compactly supported cohomology of local Shimura varieties, using the work of Fargues-Scholze. 
\end{abstract}

\maketitle

\section{Introduction}
We study the action of the Weil group on the compactly supported cohomology of local Shimura varieties using the work of Fargues-Scholze \cite{FS}. 
We first state a special case where the statement is relatively explicit and a direct analogue of the usual Eichler-Shimura relation. 

Fix a prime number $p$. 
Let $F$ be a finite extension of $\bQ_p$. 
Let $(G, b, \mu)$ a local Shimura datum over $F$ \cite{RV}*{5.1}. (We are working over $F$ rather than $\bQ_p$, but one has the associated local Shimura datum over $\bQ_p$; see \cite{RV}*{5.2}.)
Let $E\subset \overline{F}$ denote the reflex field and put $C:=\widehat{\overline{E}}$. 
For an open compact subgroup $K\subset G(F)$, let $\cM_{(G, b, \mu), K}$ denote the corresponding local Shimura variety \cite{Scholze:Berkeley}*{24.1.3}. Consider the compactly supported cohomology $R\Gamma_c (\cM_{(G, b, \mu), K}, \bZ_{\ell})$ for $\ell \neq p$ (taken after base change to $C$), equipped with left actions of the Weil group $W_E$, $J_b (F)$, and the \emph{right} action of the Hecke algebra $\bZ_{\ell}[K \backslash G(F) / K]$. 
By the usual recipe, we have a representation $r_{\mu}$ of $\widehat{G}\rtimes W_{E}$ whose highest weight is conjugate to the character corresponding to $\mu$. 
Set $d=\dim r_{\mu}$.  
The Kottwitz conjecture describes the ``discrete'' part of the cohomology using $L$-parameters composed with (the dual of) $r_{\mu}$, while we will consider the non-discrete part as well.

Assume $G$ is unramified over $F$ with a reductive model $\cG$ over $\cO_F$. Let $q$ denote the cardinality of the residue field of $F$ and $I\subset \cG(\cO_F)$ denote the the Iwahori subgroup. 
We have the polynomial, associated with the minuscule cocharacter $\mu$, 
\[
H_{G, \mu} \in \bZ[1/q][\cG(\cO_F)\backslash G(F) / \cG(\cO_F)][X]
\]
with coefficients in the Hecke algebra $\bZ[1/q][\cG(\cO_F)\backslash G(F) / \cG(\cO_F)]$; cf.~\cite{Wedhorn}*{2.6, 2.8}. (It is the one denoted by $H_{G, [\mu_E]}$ there. Actually it is only claimed that coefficients are in $\bQ[\cG(\cO_F)\backslash G(F) / \cG(\cO_F)]$ but the same argument works.). 

Regard the spherical Hecke algebra $\bC[\cG(\cO_F)\backslash G(F) / \cG(\cO_F)]$ as the ring of regular functions on irreducible unramified $\bC$-representations $\pi$ of $G (F)$. Then one can define $H_{G, \mu}$ as the function
\[
\pi \mapsto \det (X- q^{ [E:F] d/2}r_{\mu}\circ \varphi_{\pi}(\Frob_E)),
\]
where $\varphi_{\pi}\colon W_F \to \widehat{G}(\bC)\rtimes W_F$ is the $L$-parameter of $\pi$. 

The Satake isomorphism and the Bernstein isomorphism hold with coefficients in $\bZ[q^{-1/2}]$, and they are compatible. (For example, see \cite{Vigneras}*{B.4} for the Satake isomorphism, and the appendix of \cite{Haines} for the Bernstein isomorphism and the compatibility, at least with coefficients in $\bC$. These are enough to deduce the case with coefficients in $\bZ[q^{-1/2}]$.)
In particular, the compatibility gives an isomorphism 
\[
\cZ (\bZ [q^{-1/2}][I \backslash G(F) / I]) \xrightarrow[\cong]{* e_{\cG (\cO_F)}} 
\bZ[q^{-1/2}][\cG(\cO_F) \backslash G(F) / \cG (\cO_F)], 
\]
where $\cZ (-)$ denotes the center. 
Using this isomorphism, we may regard $H_{G, \mu}$ as an element of $\cZ (\bZ [q^{-1/2}][I \backslash G(F) / I])$. 
More generally, we may regard $H_{G, \mu}$ as an element of $\cZ (\bZ [q^{-1/2}][K \backslash G(F) / K])$
for any open compact subgroup $K$ containing $I$, via the map
\[
\cZ (\bZ [q^{-1/2}][I \backslash G(F) / I]) \xrightarrow{*e_{K}} 
\cZ (\bZ [q^{-1/2}][K \backslash G(F) / K]). 
\]

\begin{thm}\label{main}
Assume $K$ contains the Iwahori subgroup $I$. 
Then, for any Frobenius lift $\sigma\in W_E$, the relation
\[
E_{G, \mu} (\sigma) =0
\]
holds in $\End (R\Gamma_c (\cM_{(G, b, \mu), K}, \bZ_{\ell}[q^{1/2}]))$. 
Moreover, the action of any element $\gamma$ of the inertia subgroup $I_E$ on $R\Gamma_c (\cM_{(G, b, \mu), K}, \bZ_{\ell}[q^{1/2}])$ satisfies $ (\gamma -\id)^d=0$. 
In particular, the wild inertia subgroup $P_E$ acts trivially on $H^i_c (\cM_{(G, b, \mu), K}, \bQ_{\ell})$ for any $i$.
\end{thm}

The classical Eichler-Shimura relation, or the congruence relation, has been conjectured and proved for several Shimura varieties, e.g., \cite{Blasius-Rogawski}, \cite{Wedhorn}. 
Theorem \ref{main} is an analogue for local Shimura varieties such as unramified Rapoport-Zink spaces. 
We discuss in Section \ref{global} how Theorem \ref{main} is related to the setting of Shimura varieties. 

\begin{rem}
Let $\iota$ denote the involution
\[
\bZ_{\ell} [K \backslash G(F) / K] \ni f \xrightarrow{\iota} (KgK \mapsto f (Kg^{-1}K)) \in 
\bZ_{\ell} [K \backslash G(F) / K]^{\op}. 
\]
Using $\iota$, we can identify a left module of $\bZ_{\ell} [K \backslash G(F) / K]$ with a right module of $\bZ_{\ell} [K \backslash G(F) / K]$. 
If we restate the theorem using the \emph{left} action of the Hecke algebra, $H_{G,\mu^{-1}}$ appears instead. This is a more common formulation of the Eichler-Shimura relation in the context of Shimura varieties.  
\end{rem}

Given excursion operators constructed by Fargues-Scholze \cite{FS}, we can use Lafforgue's argument \cite{Lafforgue:main} to prove Theorem \ref{main} at least for characteristic 0 coefficients. However, we get an extra denominator $d!$ from his argument in general. 
Instead, we use the interpretation of excursion operators as the spectral action. 
For simplicity, assume $\ell$ is \emph{very good} for $\widehat{G}$ in their sense \cite{FS}*{VIII} for the moment; any $\ell$ is allowed for general linear groups, and $\ell\neq 2$ are allowed for classical groups. Fargues-Scholze constructed the spectral action \cite{FS}*{X} in this situation, working with the whole group $W_E$. 
Roughly speaking, a version of the above relation in fact holds for the vector bundle on the moduli stack of $L$-parameters associated with $r_{\mu}$, which comes with the action of $W_E$, and $R\Gamma_c (\cM_{(G, b, \mu), K}, \bZ_{\ell}[q^{1/2}])$ is obtained by applying the action of this vector bundle to $\cInd^{G(F)}_{K}\bZ_{\ell}[q^{1/2}]$. 
This implies the relation at the level of cohomology. 
We need to modify the argument slightly for a general $\ell$. 
See also Theorem \ref{Hecke} for a similar statement for more general $K$. 
When focusing on a single irreducible representation of $J_b (F)$, we can also prove a similar statement; see Remark \ref{Cayley-Hamilton 2}. 
In fact, we can prove the following result, which does not involve any choice of $\sigma \in W_E$. 

\begin{thm}\label{strong}
Let $(G, b, \mu)$ be a local Shimura datum over $F$. 
Let $\rho$ be an irreducible smooth $L$-representation of $J_b (F)$ for an algebraically closed field $L$ over $\bZ_{\ell}$. Choose $q^{1/2}\in L$. 
Let 
\[
\varphi_{\rho}^{FS}\colon W_F \to \widehat{J}_b (L) \rtimes W_F \hookrightarrow \widehat{G}(L)\rtimes W_F
\]
denote the $L$-parameter of $\rho$ with the twisted embedding \cite{FS}*{IX.4.1, IX.7.1}. 

Suppose either $\ell$ is very good for $\widehat{G}$ and $K$ is pro-$p$, or $\ell$ is invertible in $L$. 
If $\tau$ is an irreducible constituent of 
\[
\Ext^i_{J_b (F)} (R\Gamma_c (\cM_{(G, b, \mu), K}, \bZ_{\ell}[q^{1/2}]), \rho)(-d/2)   
\]
as a representation of $W_E$, 
then $\tau$ also appears in $r_{\mu^{-1}}\circ \varphi^{FS}_{\rho}$ restricted to $W_E$. 
\end{thm}

Note that Schur's lemma holds for $\rho$ and $\rho$ is admissible \cite{Vigneras:book}*{II.2.8}. 
In particular, $\varphi^{FS}_{\rho}$ above can be defined. Moreover, $R\Gamma_c (\cM_{(G, b, \mu), K}, \bZ_{\ell}[q^{1/2}])$ is a complex of smooth representations of $J_b (F)$, with a continuous action of $W_E$, by \cite{FS}*{IX.3.1} (or its proof) and the Ext modules are taken in the derived category of smooth representations and they are finite dimensional again by \cite{FS}*{IX.3.1}.  

\begin{rem}
Theorem \ref{strong} is compatible with the known description of the cohomology of Lubin-Tate and Drinfeld towers \cites{Boyer, Dat, Dat:integral}, and the Kottwitz conjecture for the discrete part in general. 
Let us also point out that the Kottwitz conjecture is often stated with an incorrect sign, and we believe our sign is correct. 

A variant of Theorem \ref{strong} holds for a non-minuscule $\mu$; cf.~\cite{FS}*{IX.3.2}. Note that it does not contradict \cite{Imai}*{8.7}. In his example, $G=\GL_2$ and $\mu =(2, 0)\in \bZ^2$, and he computed that $r_{(2,0)}\circ \varphi_{\rho}$ \emph{and} $r_{(1,1)}\circ \varphi_{\rho}$ appear (up to sign) in a non-cuspidal case.
However, $r_{(1,1)}\circ \varphi_{\rho}$ is a direct summand of $r_{(2,0)}\circ \varphi_{\rho}$ in that case. 
\end{rem}

\subsection*{Acknowledgements}
I thank Naoki Imai and Yoichi Mieda for related conversations. 
This work was supported by JSPS KAKENHI Grant Number 20K14284.

\section{Spectral action}\label{spectral action}
We recall the spectral action with integral coefficients in a general setting from \cite{FS}*{X.3}. 

Let $H$ be a split reductive group over a discrete valuation ring $\Lambda$ with an action of a finite group $Q$. 
For a finite set $I$, let $\Rep_{\Lambda}((H \rtimes Q)^I)$ denote the category of representations of $(H \rtimes Q)^I$ on finite projective $\Lambda$-modules. 
Let $W=\bZ\to Q$ be a map of groups.  
Let $Z^1 (W, H)$ denote the moduli scheme of 1-cocycles $W\to H$, equipped with the action of $H$. 
(Roughly speaking, we are considering $L$-parameters $W\to H\rtimes Q$ here.)

Suppose that we are given a $\Lambda$-linear idempotent-complete small stable $\infty$-category $\cC$, and, for finite sets $I$, 
exact $\Rep_\Lambda (Q^I)$-linear monoidal functors 
\[
\Rep_{\Lambda}((H \rtimes Q)^I) \to \End_{\Lambda} (\cC)^{BW^I}; \quad V \mapsto T_V  
\]
in a functorial way. 
It follows from \cite{FS}*{X.3.1, X.3.3} that this datum is equivalent to giving the action of 
\[
\Perf^{\Sigma} (Z^1 (W, H)/H), 
\]
the idempotent-complete stable $\infty$-subcategory of $\Perf (Z^1 (W, H)/H)$ generated by the image of $\Rep_{\Lambda} (H)$, and moreover, $T_V$ is recovered from this action via functors
\[
\Rep_{\Lambda} ((H\rtimes Q)^I) \to 
\Perf^{\Sigma} (Z^1 (W, H)/H)^{B W^I}. 
\]

Take an object $V$ of $\Rep_{\Lambda} (H\rtimes Q)$, and let $\cV$ denote the associated vector bundle on $Z^1 (W, H)/H$, with an action of $W$. We explain the Eichler-Shimura relation in this setting. 

The relation involves excursion operators \cite{FS}*{VIII}. 
Let $\gamma$ be a generator of $W$. 
Given an integer $i\geq 0$, let $S_{\wedge^i V}$ denote the following composite
\[
T_1 \to T_{\wedge^{i} V \otimes (\wedge^i V)^{\vee}}\xrightarrow{(\gamma,0)} T_{\wedge^{i} V \otimes (\wedge^i V)^{\vee}} \to T_1=\id, 
\]
which also induces an element of $\End (T_V)$. 
Here, we use the unit map $1 \to \wedge^{i} V \otimes (\wedge^i V)^{\vee}$ and the counit map $\wedge^{i} V \otimes (\wedge^i V)^{\vee}\to 1$ to define the left and right arrows, and the middle map is induced by $(\gamma,0) \in W^{\{0,1\}}$.  
It lifts to an element of $\cO (Z^1 (W, H))^H$, still denoted by $S_{\wedge^i V}$; at the level of points $W \to H$, compose the given 1-cocycle with $\wedge^i V$ and take the trace of $\gamma$. 

\begin{prop}\label{Eichler-Shimura}
If $V$ has constant rank $r$, then the following relation holds in $\pi_0\End (\cV)$:
\[
\sum_{i=0}^r (-1)^i S_{\wedge^{r-i} V} \circ \gamma^i =0.  
\]
In particular, for any object $A$ of $\cC$, 
\[
\sum_{i=0}^r (-1)^i S_{\wedge^{r-i} V} \circ \gamma^i =0. 
\]
holds in $\pi_0\End (T_V (A))$. 
\end{prop}

\begin{proof}
The first part simply follows from the Cayley-Hamilton theorem. 
The second part follows from the first by using the spectral action. 
\end{proof}

A slightly weaker statement can be proved in a more elementary manner by arguing as in \cite{Lafforgue:main}*{7.1}. 

\section{Cohomology of local Shimura varieties}
Let $G$ be a reductive group over $F$ and let $\Lambda$ be a $\bZ_{\ell}[q^{1/2}]$-algebra. 
Take a finite extension $E$ of $F$ and a finite Galois extension $E'$ of $E$ that splits $G$. 
Set $Q=\Gal (E'/E)$. 
The work of Fargues-Scholze provides \cite{FS}*{IX.2} monoidal functors
\[
\Rep_{\Lambda} ((\widehat{G}\rtimes Q)^I) \to \End_{\Lambda} (\cD_{\lis} (\Bun_G, \Lambda))^{BW_E^I}; \quad V \mapsto T_V
\]
functorial in finite sets $I$, where $\cD_{\textnormal{lis}} (\Bun_G, \Lambda)$ is a condensed $\infty$-category defined there. (To be precise, this version is not stated explicitly; $Q$ there is a quotient of $W_F$, not $W_E$, in our notation. See also the discussion in \cite{FS}*{IX.3}.)
In particular, fixing an element $\gamma\in W_E$, we obtain
\[
\Rep_{\Lambda} ((\widehat{G}\rtimes Q)^I) \to \End_{\Lambda} (\cD_{\textnormal{lis}} (\Bun_G, \Lambda))^{B(\gamma^{\bZ})^I}
\]
to which we can apply Proposition \ref{Eichler-Shimura}. 

On the other hand, any representation $\widehat{G}\rtimes W_E \to \GL (V)$ and a choice of an element $\gamma \in W_E$ give rise to an element of $\cO (Z^1 (\gamma^{\bZ}, \widehat{G})_{\Lambda})^{\widehat{G}}$, ``excursion operators'', by taking the trace as before. 
Fargues-Scholze also constructed a map
\[
\cO (Z^1 (\gamma^{\bZ}, \widehat{G})_{\Lambda})^{\widehat{G}} \to \cZ (G(F), \Lambda)
\]
to the Bernstein center of smooth $\Lambda$-representations of $G(F)$. 
Composing it with the natural projection
\[
\cZ (G(F), \Lambda) \to \cZ (\End_{G(F)}(\cInd^{G(F)}_K \Lambda))=\cZ (\Lambda [K \backslash G(F) / K]), 
\]
we see that the pair $V, \gamma$ gives an element of $\cZ (\Lambda [K \backslash G(F) / K])$. 
Moreover, this is precisely the excursion operator $S_{V, \gamma}$ acting on $\cInd^{G(F)}_K \Lambda$:
\[
\cInd^{G(F)}_K \Lambda \to T_{V\otimes V^{\vee}} (\cInd^{G(F)}_K \Lambda) \xrightarrow{(\gamma, e)} T_{V\otimes V^{\vee}} (\cInd^{G(F)}_K \Lambda) \to \cInd^{G(F)}_K \Lambda. 
\]
See \cite{FS}*{IX.5}; we directly use excursion operators and do not need the spectral Bernstein center, so the additional assumption there can be ignored. 

\begin{thm}\label{Hecke}
Let $(G, b, \mu)$ be a local Shimura datum over $F$ with the reflex field $E$. 
For an open compact subgroup $K\subset G(F)$ and any element of $\gamma \in W_E$, the action of $\gamma$ on 
\[
R\Gamma_{c} (\cM_{(G, b, \mu), K}, \bZ_{\ell}[q^{1/2}])(d/2)
\]
satisfies the relation 
\[
\sum_{i=0}^d (-1)^i S_{\wedge^{d-i} r_{\mu}, \gamma} \circ \gamma^i =0, 
\]
where $d=\dim r_{\mu}$. 
\end{thm}

Note that the representation $r_{\mu}$ can be defined with integral coefficients in a unique way. 

\begin{proof}
As in the proof of \cite{FS}*{IX.3.1}, there is an identification
\[
R\Gamma_c (\cM_{(G, b, \mu), K}, \bZ_{\ell}[q^{1/2}]) [d](d/2) \cong
i^{b*} T_{r_{\mu}} (j_! \cInd^{G(F)}_{K}\bZ_{\ell}[q^{1/2}])
\]
in $D_{\textnormal{lis}}(\Bun^b_G, \bZ_{\ell}[q^{1/2}])\cong D ({J_b (F)}, \bZ_{\ell}[q^{1/2}])$ compatibly with actions of $W_E$, 
where $j\colon \Bun^1_G \to \Bun_G, i^b \colon \Bun^b_G \to \Bun_G$ are immersions. 
(The shift and twist come from the normalization of perverse sheaves that is used to define $T_{r_{\mu}}$.)
This identification is also compatible with actions of the Hecke algebra if we consider the \emph{right} action on $i^{b*} T_{r_{\mu}} (j_! \cInd^{G(F)}_{K}\bZ_{\ell}[q^{1/2}])$. 

Therefore, it suffices to study the action of $\gamma$ on $T_{r_{\mu}} (j_! \cInd^{G(F)}_{K}\bZ_{\ell}[q^{1/2}])$. By Proposition \ref{Eichler-Shimura}, 
\[
\sum_{i=0}^d (-1)^i S_{\wedge^{d-i}r_{\mu}, \gamma}\circ \gamma^i =0
\]
holds in $\pi_0\End (T_{r_{\mu}} (j_! \cInd^{G(F)}_{K}\bZ_{\ell}[q^{1/2}]))$. 
As excursion operators commute with the Hecke operator $T_{r_{\mu}}$, the action of $S_{\wedge^{d-i} r_{\mu}, \gamma}$ on $T_{r_{\mu}} (j_! \cInd^{G(F)}_K\bZ_{\ell}[q^{1/2}])$ comes from that on $j_! \cInd^{G(F)}_{K}\bZ_{\ell}[q^{1/2}]$. 
This completes the proof.
\end{proof}

\begin{proof}[Proof of Theorem \ref{main}]
Given Theorem \ref{Hecke}, it suffices to determine the image of $S_{\wedge^{d-i} r_{\mu}, ?}$ for $?=\sigma, \gamma$. 
For this purpose, we are allowed to invert $\ell$, and work with the Hecke algebra $\overline{\bQ}_{\ell} [\cG(\cO_F)\backslash G(F) / \cG(\cO_F)]$. 

Regarding the spherical Hecke algebra as the ring of regular functions on irreducible unramified $\overline{\bQ}_{\ell}$-representations $\pi$, we need only compute the value of $S_{\wedge^{d-i} r_{\mu}, ?}$ at each $\pi$. 

Now, we can use $L$-parameters $\varphi_{\pi}^{FS}$ of Fargues-Scholze \cite{FS}*{IX.4.1}. 
First note that, as $\pi$ is unramified, $\varphi_{\pi}^{FS}$ is the same as the usual $L$-parameter $\varphi_{\pi}$ \cite{FS}*{IX.6.4, IX.7.3}.  
By the definition of $\varphi_{\pi}^{FS}$, $S_{\wedge^{d-i} r_{\mu}, \sigma}(\pi)$ is equal to the trace of 
$(\wedge^{d-i} r_{\mu}\circ \varphi_{\pi}^{FS})(\sigma)$. 
This is the coefficient of $H_{G,\mu}$ after taking into account of the Tate twist.   
Similarly, we can compute $S_{\wedge^{d-i} r_{\mu}, \gamma}(\pi)$: as $\varphi_{\pi}^{FS}$ is unramified, $\gamma$ acts trivially on $\wedge^{d-i} r_{\mu}\circ \varphi_{\pi}^{FS}$. 
Therefore, $S_{\wedge^{d-i} r_{\mu}, \gamma}(\pi)=\dim \wedge^{d-i} r_{\mu}$. 

For the claim about the action of $P_E$, use that the action factors through a finite quotient of $P_E$ \cite{FS}*{IX.5.1}. 
\end{proof}

\section{Complements on mod $\ell$ representations}
We make some remarks on $\overline{\bF}_{\ell}$-representations related to Theorem \ref{main}. 
We choose $q^{1/2}\in \overline{\bF}_{\ell}$; cf.~\cite{Treumann-Venkatesh}*{7.4}. 
Vign\'eras attached to an irreducible smooth representation $\pi$ with nonzero Iwahori fixed vector a character of the Hecke algebra $\overline{\bF}_{\ell}[\cG(\cO_F)\backslash G(F)/ \cG(\cO_F)]$ by looking at the action on Iwahori fixed vectors \cite{Vigneras}*{B.5} (this is unconditional as \cite{Vigneras}*{(B.2.5)} is known \cite{Dat:finitude}*{proof of 6.3}), equivalently, via a less direct way through \cite{Vigneras}*{B.3, B.4}. 
Such a character corresponds to a semisimple (twisted) $\widehat{G}(\overline{\bF}_{\ell})$-conjugacy class of $\widehat{G}(\overline{\bF}_{\ell})\rtimes \Frob_{F}$ \cite{Treumann-Venkatesh}*{(7.4.5)}. Namely, any irreducible smooth representation $\pi$ with nonzero Iwahori fixed vector gives rise to an unramified $L$-parameter
\[
\varphi_{\pi}^{TV}\colon W_F \to \widehat{G}(\overline{\bF}_{\ell})\rtimes W_F. 
\]
(Different unramified representations may give rise to the same $L$-parameter \cite{Vigneras}*{B.4}.)

On the other hand, one of the main results of Fargues-Scholze \cite{FS}*{IX.4.1} also gives an $L$-parameter
\[
\varphi_{\pi}^{FS}\colon W_F \to \widehat{G}(\overline{\bF}_{\ell})\rtimes W_F. 
\]

Two parameters $\varphi_\pi^{TV}$, $\varphi_{\pi}^{FS}$ are the same : both construction commute with parabolic induction in a suitable sense \cite{Treumann-Venkatesh}*{7.5}, \cite{FS}*{IX.7.3}, and it reduces to the case of unramified tori. This case can be settled by \cite{Treumann-Venkatesh}*{2.9}, \cite{FS}*{IX.6.4}. 

The reduction mod $\ell$ of $H_{G,\mu}$ can be also regarded as a function
\[
\pi \mapsto \det (X- q^{[E:F] d/2}r_{\mu}\circ \varphi_{\pi}^{TV}(\Frob_E))=\det (X- q^{[E:F] d/2}r_{\mu}\circ \varphi_{\pi}^{FS} (\Frob_E))
\]
on irreducible unramified $\overline{\bF}_{\ell}$-representations. 

\section{Proof of Theorem \ref{strong}}
As in the statement, let $(G, b, \mu)$ be a local Shimura datum over $F$, and take an open compact subgroup $K\subset G(F)$. 
Let $\rho$ be an irreducible smooth $L$-representation of $J_b (F)$ with an algebraically closed $L$ over $\bZ_{\ell}$. 
Choose $q^{1/2}\in L$. 

As in \cite{FS}*{p.315}, we see that
\[
R\Hom_{J_b (F)} (R\Gamma_c (\cM_{(G, b, \mu), K}, \bZ_{\ell}[q^{1/2}]), \rho) (-d/2)
\]
can be obtained from
\[
j^* T_{r_{\mu^{-1}}} Ri^b_* \rho
\]
by taking $K$-invariants. 
So, after replacing $\mu$, it suffices to show that  
$j^* T_{r_{\mu}}Ri^b_* \rho$ satisfies the corresponding statement. 
(As it is a complex of admissible representations of $G(\bQ_p)$, only finite dimensional $W_E$-representations appear.)

Now we shall use the spectral action, working with the whole group $W_E$. (Compare with Section \ref{spectral action}.)
Set $\Lambda=\bZ_{\ell}[q^{1/2}]$. 
Let $Z^1 (W_E, \widehat{G})_{\Lambda}$ denote the scheme parametrizing $L$-parameters \cite{FS}*{VIII}, with $\widehat{G}$-action and the quotient map
\[
Z^1 (W_E, \widehat{G}) \to Z^1 (W_E, \widehat{G}) \sslash G
\]
to the coarse moduli. 
By construction, $\varphi_{\rho}^{FS}$, well-defined up to $\widehat{G}(L)$-conjugacy, is an $L$-valued point of $Z^1 (W_E, \widehat{G}) \sslash G$. 

Assuming $\ell$ is very good for $\widehat{G}$, Fargues-Scholze \cite{FS}*{X.0.1} prove that the excursion operators
\[
\Rep_{\Lambda} ((\widehat{G}\rtimes Q)^I) \to \End_{\Lambda} (\cD_{\lis} (\Bun_G, \Lambda))^{BW_E^I} 
\]
produce the action of 
\[
\Perf (Z^1 (W_E, \widehat{G})_{\Lambda} / \widehat{G})
\]
on $\cD_{\lis} (\Bun_G, \Lambda)$, and the excursion operators are recovered by composing functors 
\[
\Rep_{\Lambda} ((\widehat{G}\rtimes Q)^I) \to \Perf (Z^1 (W_E, \widehat{G})_{\Lambda} / \widehat{G})^{BW_E^I}
\]
with this action. 
Also, the actions on $i^b_! \rho, i^b_! \rho^{\vee}$ factor through 
\[
\Perf (Z^1 (W_E /P, \widehat{G})_{\Lambda} / \widehat{G})
\]
for some open compact subgroup $P\subset W_E$ of the wild inertia subgroup \cite{FS}*{IX.5.1}, and the same holds true for $Ri^b_* \rho\cong \mathbb{D} (i^b_! \rho^{\vee})$ by \cite{FS}*{IX.2.2}.  
In particular, if we write $\cV_{r_{\mu}}$ for the vector bundle on $Z^1 (W_E /P, \widehat{G})_{\Lambda} / \widehat{G}$ corresponding to $r_{\mu}$, $T_{r_{\mu}} Ri^b_* \rho$ is obtained by applying the action of $\cV_{r_{\mu}}$ to $Ri^b_* \rho$. 
We will study $W_E \to \End (\cV_{r_\mu})$. 
Using canonical isomorphisms
\[
\End (Ri^b_* \rho)\cong \End (i^b_! \rho)\cong \End_L (\rho) \cong L
\]
and compatibility of excursion operators and the spectral action, 
we have the following commutative diagram:
\[
\begin{CD}
\cO (Z^1 (W_E/ P, \widehat{G})_{\Lambda})^{\widehat{G}} @>>> \pi_{0}\End (\cV_{r_{\mu}}) \\
@VVV @VVV \\
\End_L (\rho) @>>> \pi_{0}\End (T_{r_{\mu}} Ri^b_* \rho), 
\end{CD}
\]
where $\cO (Z^1 (W_E/ P, \widehat{G})_{\Lambda})^{\widehat{G}}$ is the ring of regular functions on the affine scheme $Z^1 (W_E/ P, \widehat{G})_{\Lambda} \sslash \widehat{G}$, and the left arrow corresponds to $\varphi^{FS}_{\rho}$.

If $\varphi\colon W_E \to \widehat{G}(L)\rtimes W_E$ is an $L$-parameter and it maps to $\varphi_{\rho}^{FS}$ under the quotient map, let us say that the semisimplification of $\varphi$ is $\varphi_{\rho}^{FS}$. 
The semisimplification of the representation $r_{\mu}\circ \varphi_{\rho}$ is constant for such $\varphi$. 

Take an irreducible constituent $\tau$ of $H^i (j^* T_{r_{\mu}}i^b_! \rho)$ for some $i$ regarded as a 
representation of $W_E$, and assume that $\tau$ does not appear in $r_{\mu}\circ \varphi^{FS}_{\rho}$. 
(As remarked before, $\tau$ is finite dimensional.)
Let $\tau_1, \dots, \tau_n$ denote the irreducible constituents of $r_{\mu}\circ \varphi^{FS}_{\rho}$.  
Then, the action map 
\[
L[W_E]\to \End_L (\tau) \times \prod_i \End_L (\tau_i)
\]
is surjective. (This is a version of Burnside's theorem, and follows from Jacobson's density and Schur's lemma as usual.)
Take an element $e_{\tau}\in L [W_E]$ that maps to $(\id_{\tau}, 0, \dots, 0)$. 
Then, the action of $e_{\tau}$ on $ r_{\mu}\circ \varphi$ is nilpotent for any $L$-parameter $\varphi$ with the semisimplification $\varphi^{FS}_{\rho}$. In particular, the action of $e_{\tau}^d$ is 0 for $d=\dim r_{\mu}$.  
Therefore the image of $e_{\tau}^d$ in 
\[
\End_L (\rho)\otimes_{\cO (Z^1 (W_E/ P, \widehat{G})_{\Lambda})^{\widehat{G}}} \pi_{0}\End (\cV_{r_{\mu}})
\]
is $0$. 
On the other hand, $e_{\tau}^d$ acts non-trivially on $j^* T_{r_{\mu}} Ri^b_* \rho$, thus these contradict. 
We completed the proof of Theorem \ref{strong}. 

\begin{rem}\label{Cayley-Hamilton 2}
Using the Cayley-Hamilton theorem applied to $\cV_{r_\mu}$ as before, we can prove that the action of $\gamma\in W_E$ on
\[
R\Hom_{J_b (F)} (R\Gamma_c (\cM_{(G, b, \mu), K}, \bZ_{\ell}[q^{1/2}]), \rho) (-d/2)
\]
is killed by the characteristic polynomial of $(r_{\mu^{-1}}\circ \varphi_{\rho})(\gamma)$. 
This is closer to the original Eichler-Shimura relation. 
While it is clearly compatible with Theorem \ref{strong}, it does not seem to imply Theorem \ref{strong}.  
\end{rem}

Let us record some variants. 
Let $\pi$ be an irreducible smooth $L$-representation of $G(F)$. 
The same argument applied to $i^{b*} T_{r_{\mu}}j_* \pi$ and the duality isomorphsim at the infinite level imply:

\begin{cor}\label{switch}
Let $\pi$ be an irreducible smooth $L$-representation of $G(F)$. 
If $\tau$ is an irreducible constituent of
\[
\Ext^i_{G(F)} (\varinjlim_K R\Gamma_c (\cM_{(G, b, \mu), K}, \bZ_{\ell}[q^{1/2}]), \pi)(-d/2), 
\]
then $\tau$ appears in $r_{\mu}\circ \varphi^{FS}_{\pi}$. 
\end{cor}

One can also consider $\pi$ and $\rho$ simultaneously. By the argument in \cite{FS}*{p.315}, 
\[
\Ext^i_{G(F)\times J_b (F)} (\varinjlim_K R\Gamma_c (\cM_{(G, b, \mu), K}, \bZ_{\ell}[q^{1/2}]), \pi \boxtimes \rho)
\]
is finite dimensional. 

\begin{cor}\label{common}
If $\tau$ is an irreducible constituent of
\[
\Ext^i_{G(F)\times J_b (F)} (\varinjlim_K R\Gamma_c (\cM_{(G, b, \mu), K}, \bZ_{\ell}[q^{1/2}]), \pi \boxtimes \rho)(-d/2)
\]
for some $i$, then $\tau$ appears both in $r_{\mu}\circ \varphi^{FS}_{\pi}$ and $r_{\mu^{-1}}\circ \varphi^{FS}_{\rho}$.  
\end{cor}

\begin{rem}
As mentioned above, excursion operators acts on
\begin{multline*}
R\Hom_{J_b (F)} (\varinjlim_K R\Gamma_c (\cM_{(G, b, \mu), K}, \bZ_{\ell}[q^{1/2}]), \rho)(-d/2) \\
\cong 
j^* T_{r_{\mu^{-1}}}(Ri_*^b \rho) \in D(G(\bQ_p), L) 
\end{multline*}
via the character $\cO (Z^1 (W_E/ P, \widehat{G})_{\Lambda})^{\widehat{G}}\to L$ corresponding to $\varphi_{\rho}$. 
Thus, if  
\[
\Ext^i_{G(F)\times J_b (F)} (\varinjlim_K R\Gamma_c (\cM_{(G, b, \mu), K}, \bZ_{\ell}[q^{1/2}]), \pi \boxtimes \rho)(-d/2)
\]
is nonzero for some $i$, then $\pi^{\vee}$ and $\rho$ have the same $L$-parameter of Fargues-Scholze. 
Let us also recall that taking the smooth dual of irreducible representations is compatible with applying the Chevalley involution of $\widehat{G}$ to $L$-parameters \cite{FS}*{IX.5.3}. 
\end{rem}

\section{Relation to Shimura varieties}\label{global}
We briefly explain how Theorem \ref{main} is related to the Eichler-Shimura relation for Shimura varieties. 
Let $(G, X)$ be a Shimura datum with the reflex field $E$ and assume that $G_{\bQ_p}$ is unramified with a hyperspecial open compact subgroup $K_p$. Take a sufficiently small open compact subgroup $K=K^p K_p \subset G(\bA_f)$. 
Let $\Sh (G, X)_K$ denote the associated Shimura variety of level $K$ with dimension $d$, defined over $E$. 
(We use the convention that the usual moduli problem gives rise to the canonical model in the PEL case.)
Choose $q^{1/2}\in \overline{\bQ}_{\ell}$. Fix a finite place $v$ of $E$ dividing $p$ and a compatible embedding $\overline{\bQ}\hookrightarrow \overline{\bQ}_p$. 
Under a suitable assumption, Mantovan's formula gives the following equality
\begin{multline*}
\sum_{i=0}^{2d} (-1)^i H^i_c (\Sh (G, X)_K, \overline{\bQ}_{\ell})(d/2) \\
= 
\sum_{b \in B(G_{\bQ_p}, \{\mu\})} \sum_{j\geq 0} (-1)^{j} \Ext^{-2d+j}_{J_b (\bQ_p)} (R\Gamma_c (\cM_{(G_{\bQ_p}, b, \mu), K_p}, \overline{\bQ}_{\ell}), R\Gamma_c (\textnormal{Ig}^b, \overline{\bQ}_{\ell}))(-d/2) 
\end{multline*}
in the Grothendick group of continuous representations of $W_{E_v}$, where $\textnormal{Ig}^b$ denotes the Igusa variety. 
(The action of $W_{E_v}$ on $R\Gamma_c (\textnormal{Ig}^b, \overline{\bQ}_{\ell})$ is trivial.)
In this sense, Theorem \ref{main} is compatible with the Eichler-Shimura relation for $\Sh (G, X)_{K}$ when the Mantovan's formula holds. 
(Note that the image of $E_v\hookrightarrow \overline{\bQ}_p$ is the reflex field of the local Shimura datum.)
In fact, known proofs of Mantovan's formula use excision sequences with respect to the Newton stratification. 
(When we mention Mantovan's formula below, we include such an argument.)
Therefore, one can actually deduce from Theorem \ref{main}, using the convention that Hecke algebra acts by left, 
 that
\begin{cor}
If Mantovan's formula holds, then the action of $\Frob_v$ on 
\[
H^i_c (\Sh (G, X)_K, \overline{\bQ}_{\ell})
\]
is killed by \emph{some power} of $H_{G_{\bQ_p},\mu^{-1}}$. 
\end{cor}

Mantovan's formula itself actually can hold for more general level $K_p$, and we can conclude a more precise result from Corollary \ref{switch}.   
(See \cite{Hamacher-Kim} for cases where Mantovan's formula holds.)
So, we allow general $K_p$ from now on. 
Assume $\Sh (G, X)_K$ are proper over $E$, for simplicity, and, using Matsushima's formula, 
write
\[
\varinjlim_{K\subset G(\bA_f)} H^i (\Sh (G, X)_K, \overline{\bQ}_{\ell})(d/2)=\bigoplus_{\pi_f} \pi_f \otimes  H^i (\pi_f), 
\]
where $\pi_f$ runs through irreducible admissible representations of $G(\bA_f)$, and $H^i (\pi_f)$ is a representation of $\Gal (\overline{\bQ}/E)$. 
We can deduce from Corollary \ref{switch} and (the proof of) Mantovan's formula that 

\begin{thm}
Assume that $\Sh (G, X)_K$ are proper over $E$ and Mantovan's formula holds true. If $\tau$ is an irreducible constituent of $H^i (\pi_f)$ as a representation of $W_{E_v}$, then $\tau$ appears in $r_{\mu^{-1}}\circ \varphi_{\pi_p}$. 
\end{thm}

\end{document}